\documentclass[a4paper,11pt]{article}
\usepackage[utf8]{inputenc}
\usepackage{amsmath,amsthm,amssymb}
\usepackage{graphicx}
\usepackage[all]{xy}
\usepackage{verbatim}
\usepackage{mathrsfs}
\usepackage{enumerate}
\usepackage{float}
\usepackage{pgf}
\usepackage{tikz}
\usetikzlibrary{trees,graphs}
\usetikzlibrary{arrows,decorations.markings,patterns,calc,automata}

\usepackage{Alegreya}
\usepackage{euler}

\DeclareMathOperator{\Ker}{Ker}
\DeclareMathOperator{\Coker}{Coker}

\theoremstyle{plain}
\newtheorem{theorem}{Theorem}[section]
\newtheorem{lemma}[theorem]{Lemma}
\newtheorem{corollary}[theorem]{Corollary}
\newtheorem{proposition}[theorem]{Proposition}
\theoremstyle{definition}
\newtheorem{remark}[theorem]{Remark}
\newtheorem{notation}[theorem]{Notation}
\newtheorem{definition}[theorem]{Definition}
\newtheorem{example}[theorem]{Example}

\newcommand{\tree}{\mathcal T}
\newcommand{\vertices}{\mathcal V}
\newcommand{\edges}{\mathcal E}
\newcommand{\tfg}{\mathcal F}

\newcommand{\Aut}{\operatorname{Aut}}
\newcommand{\Fix}{\operatorname{Fix}}
\newcommand{\Sym}{\operatorname{Sym}}

\newcommand{\ccol}{\operatorname{ccol}}
\newcommand{\prm}{\operatorname{prm}}
\newcommand{\lpc}{\operatorname{lpc}}

\newcommand{\AAut}{\operatorname{AAut}}
\newcommand{\frakg}{\mathfrak{g}}
\newcommand{\gpoid}{\mathcal{G}}

\usepackage{titlefoot}

\begin{document}

\title{Topological full groups and t.d.l.c. completions of Thompson's~$V$}
\author{Waltraud Lederle \\ \scriptsize{Ben Gurion University of the Negev}}
\date{}
\maketitle\unmarkedfntext{This work was partially supported by Israel Science Foundation grant ISF 2095/15.}

\begin{abstract}
We show how all topological full groups coming from a one-sided irreducible shift of finite type, as studied by Matui, can be re-interpreted as groups of colour-preserving tree almost automorphisms.
As an application, we show that they admit t.d.l.c. completions of arbitrary finite local prime content.
This applies in particular to Thompson's $V$. 
\end{abstract}

\section{Introduction}

In 1965 Thompson constructed groups to describe rules for rearranging formal expressions, which are now commonly called Thompson's $F, T$ and $V$.
The groups $T$ and $V$ are simple and finitely presented; they were the first infinite groups with these properties.
Higman \cite{h74} generalized Thompson's group $V$ by introducing algebras of type $n$, the extension of these generalizations to $T$ and $F$ were first considered by Brown \cite{b87},
and all are now commonly known as the Higman--Thompson groups $V_{d,k}, T_{d,k}$ and $F_{d,k}$.
Since then many other generalizations have occurred, for example "higher dimensional" versions by Brin.
The generalization we focus on in the present paper is by Matui, where $V$ is viewed as the topological full group of a groupoid associated to a one-sided irreducible shift of finite type.

Caprace and De Medts \cite{cm11} observed that the Higman-Thompson groups $V_{d,k}$ embed densely into groups of tree almost automorphisms, which are totally disconnected, locally compact (t.d.l.c.) groups. Typically these groups are denoted $\AAut_D(\tree_{d,k})$ in the literature. More precisely, they found one completion of $V_{d,k}$ for every conjugacy class of subgroups of $\Sym(d)$, but it is unclear whether some of these completions are isomorphic.
It is often a non-trivial task to determine for t.d.l.c. groups whether they are isomorphic, or whether they have isomorphic compact open subgroups or not. One method is to investigate for which prime numbers $p$ the group contains an infite pro-$p$ subgroup. This set of primes is called the \emph{local prime content} and was introduced by Gl\"ockner \cite{g06}.

\begin{definition}
	Let $H$ be a (topological) group. A \emph{completion} of $H$ is a locally compact group $G$ together with a (continuous) group homomorphism $f \colon H \to G$ with dense image. The completion is called t.d.l.c. if $G$ is a t.d.l.c. group.
\end{definition}

The following question was suggested to me by Pierre-Emmanuel Caprace. Do all topological full groups coming from one-sided shifts, in the sense of Matui, admit a t.d.l.c. completion? In the present work we give a positive answer.

\begin{theorem}[see Theorem \ref{thm:mainthm}]
	Let $P$ be a finite set of primes. Every topological full group of the groupoid associated to a one-sided irreducible shift of finite type has a t.d.l.c. completion with local prime content $P$.
\end{theorem}

The proof idea of the theorem is the following. In previous work \cite{l17} we showed how certain groups of tree almost automorphisms can be written as topological full groups coming from one-sided shifts in the sense of Matui \cite{m15}.
In Theorem \ref{thmshifttoalmaut} we show that the reverse is always possible, i.e. for every irreducible one-sided shift of finite type there exists a tree with an appropriate colouring such that the topological full group of the shift is the group of colour-preserving tree almost automorphisms. Results by Matsumoto allow us to find different shifts with the same topological full group. By rewriting them as tree almost automorphism groups we find different completions of the same group.

\begin{remark}
	It is a result by Matui that all topological full groups coming from one-sided shifts have simple commutator subgroup. We show that the closures of the commutator subgroups in the completions we find are topologically simple, see Proposition \ref{prop:simple}. It would be interesting to know whether all completions of topological full groups need to be topologically, or even abstractly, simple.
\end{remark}

\paragraph*{}
Le Boudec and Wesolek \cite{lbw16} investigated commensurability classes of commensurated subgroups, which are closely connected to t.d.l.c. completions (see e.g. Reid and Wesolek \cite{rw15}), of tree almost automorphism groups of regular trees. For Thompson's $F$ and $T$ they completely determine the commensurability classes of commensurated subgroups, showing that none of these groups allows a non-discrete t.d.l.c. completion, see \cite{lbw16} Section 6.2. For $V$ they find the following restriction.

\begin{proposition}[\cite{lbw16}, Proposition 6.4]
	Let $H \leq V$ be a proper commensurated subgroup of $V$. Then $H$ is locally finite, i.e. every finitely generated subgroup of $H$ is finite.
\end{proposition}

Le Boudec and Wesolek leave open, Question 6.5 and Question~6.6 in \cite{lbw16}, whether Thompson's $V$ has more than three commensurability classes of commensurated subgroups, or whether it has more than one non-discrete locally compact completion. We answer both questions.

\begin{corollary}[see Example \ref{ex:v}]
	Thompson's group $V$ has at least countably many commensurability classes of commensurated subgroups.
	Moreover, there are at least countably many locally non-isomorphic groups of tree almost automorphisms into which $V$ embeds densely.
\end{corollary}

It would be interesting to know whether there are any other commensurability classes of commensurated subgroups or any other completions of $V$ than what we find in the present work, especially whether there are any completions which cannot be written as tree almost automorphism groups.

\begin{remark}
	It is a classical fact that $V$ is finitely generated and simple. As a consequence it is not residually finite, hence it is not a linear group over $\mathbb{C}$.
	Therefore it follows from work by Boyko, Gefter and Kulagin \cite{bgk03}, Theorem 9, that every completion of $V$ has to be totally disconnected.
\end{remark}

\section{Preliminaries}

The following convention is used throughout the paper.

\paragraph*{Convention.} A graph is pair $\frakg=(\mathfrak{v},\mathfrak{e})$ together with two maps $o,t \colon \mathfrak{e} \to \mathfrak{v}$ called origin and terminus. In particular, all graphs are oriented.
All trees are in addition assumed to be rooted, which allows to talk about children, parents and descendants of vertices. The orientation of a tree is poining away from the root, i.e. for each edge $e$ of a tree $t(e)$ is a child of $o(e)$.

\subsection{Tree almost automorphism groups}
\label{ch:almaut}

Throughout this section $\tree=(\vertices,\edges)$ denotes an arbitrary locally finite tree with root $v_0$.

\subsubsection{Definition of almost automorphisms} \label{sectalmautdef}

\begin{definition}
	A finite subtree $T \subset \tree$ is called \index{complete} \emph{complete} if it contains the root $v_0$,
	and for each vertex $w$ it contains either all or none of the children of $w$.
\end{definition}

\begin{notation} \label{nottreev} Throughout the paper we will use the following.
	\begin{itemize}
		\item For a subtree $T \subset \tree$ we denote by \index{leaf} $\mathcal{L}T \subset \mathcal V$ the set of leaves of $T$.
		\item For every vertex $v \in \mathcal V$ we denote by $\tree_v$ the subtree of $\tree$ with root $v$ and
		whose vertices are all the descendants of $v$.
		\item For a finite complete subtree $T \subset \tree$ the difference $\tree \setminus T$ will always denote the subgraph $\bigsqcup_{v \in \mathcal{L}T} \tree_v$ of $\tree$.
		Hence $\tree \setminus T$ is a forest with $|\mathcal{L}T|$ many connected components.
	\end{itemize}
\end{notation}

\begin{definition}
	Let $T_1$ and $T_2$ be finite complete subtrees of $\tree$.
	An \index{honest almost automorphism} \emph{honest almost automorphism} of $\tree$ is a forest isomorphism $\varphi \colon \tree \setminus T_1 \to \tree \setminus T_2$. 
\end{definition}
%

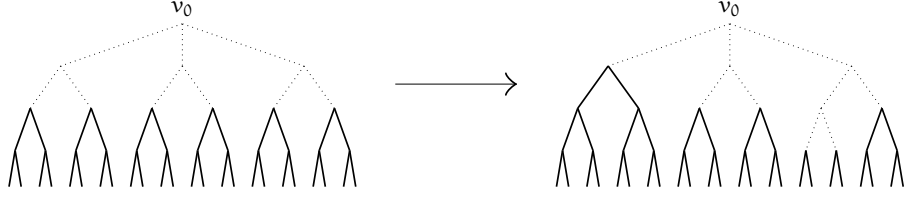
\begin{figure}[H]
	\centering
	\scalebox{0.8}{
		\begin{tikzpicture} [grow'=down]
		
		\tikzstyle{level 1}=[level distance=7mm,sibling distance=20mm]
		\tikzstyle{level 2}=[level distance=7mm,sibling distance=10mm]
		\tikzstyle{level 3}=[level distance=7mm,sibling distance=5mm]
		\tikzstyle{level 4}=[level distance=6mm,sibling distance=2mm]
		\tikzstyle{level 5}=[level distance=5mm,sibling distance=1mm]
		
		\node [inner sep=0pt,outer sep=-0.4pt] at (-4.5,0) {}
		child[dotted] foreach \i in {1,2,3}
		{child {child[solid,thick] {child  child  } 
				child[solid,thick] {child  child  } } 
			child {child[solid,thick] {child  child  } 
				child[solid,thick]  {child child  } } }; 
		\node [inner sep=0pt,outer sep=-0.4pt] at (4.5,0) {}
		child[dotted]
		{child {child[solid,thick] {child  child  } child[solid,thick] {child  child  } } 
			child {child {child[solid,thick]  child[solid,thick]  } child {child[solid,thick]  child[solid,thick]  } } }
		child[dotted]
		{child {child[solid,thick] {child  child  } child[solid,thick] {child  child  } } 
			child {child[solid,thick] {child  child  } child[solid,thick] {child  child  } } }
		child[dotted]
		{ child[solid,thick] {child {child  child  } child {child  child  } } 
			child[solid,thick] {child {child  child  } child {child  child  } } };
		
		\node [above] at (-4.5,0) {$v_0$};
		\node [above] at (4.5,0) {$v_0$};
		;
		
		\draw[decoration={markings,mark=at position 1 with {\arrow[line width=0.1mm,scale=3.5]{>}}},postaction={decorate}] (-1,-1)--(1,-1);    
		\end{tikzpicture}
	}
	\caption{The trees $T_1$ and $T_2$ are indicated with dotted lines.}
	\label{almautfigure}
\end{figure}
%

\paragraph*{Equivalence relation on honest almost automorphisms.} 
Let $T_1,T_2,T'_1,T'_2 \subset \tree$ be finite complete subtrees. Let
$\varphi \colon  \tree \setminus T_1 \to \tree \setminus T_2$ and $\psi \colon  \tree \setminus T'_1 \to \tree \setminus T'_2$ be honest almost automorphisms of $\tree$. 
We say that $\varphi$ and $\psi$ are \index{equivalent almost automorphisms} \emph{equivalent} if there exists a finite complete subtree $T \supset T_1 \cup T'_1$
such that $\varphi|_{\tree \setminus T} = \psi|_{\tree \setminus T}$.

\begin{definition}
	With this equivalence relation, an \index{almost automorphism} \emph{almost automorphism} of $\tree$ is the equivalence class of an honest almost automorphism.
\end{definition}

We will usually not distinguish between an honest almost automorphisms and its equivalence class,
but say it explicitly whenever we need to talk about an honest almost automorphism.

In proofs it will be convenient to work with a certain generating set for this equivalence relation. We need the following notion.

\paragraph*{Simple expansions.}
For finite complete subtrees $T \subset T' \subset \tree$ and a leaf $v$ of $T$, we say that $T'$ is obtained from $T$ by a \index{simple expansion} \emph{simple expansion} at $v$, or simply $T'$ is a simple expansion of $T$, if $T'$ is spanned by $T$ and all children of $v$. Note that any finite complete subtree of $\tree$ containing $T$ is obtained from $T$ by a sequence of simple expansions.
If in the definition of the equivalence relation on honest almost automorphisms above we require that $T'_1$ is obtained from $T_1$ by a simple expansion and set $T=T'_1$,
the resulting relation generates the equivalence relation. 

\begin{remark} \label{equivclassrmk}
	Let $T_1,T_2$ be finite complete subtrees of $\tree$ and let $\varphi \colon \tree \setminus T_1 \to \tree \setminus T_2$ be an honest almost automorphism.
	Then, for every finite complete subtree $T$ of $\tree$ containing $T_1$ there exists a unique finite complete subtree $T'$ of $\tree$ containing $T_2$ and a unique representative $\psi \colon \tree \setminus T \to \tree \setminus T'$ of $\varphi$.
	The analogous statement holds for $T \supset T_2$.
\end{remark}

If we pass to equivalence classes, the preceding remark allows us to define a product on the set of all almost automorphisms.

\paragraph*{Composition of two almost automorphisms.} \index{product of almost automorphisms} Let $T_1,\dots,T_4 \subset \tree$ be finite complete subtrees of $\tree$.
Let $\varphi \colon \tree \setminus T_1 \to \tree \setminus T_2$ and $\psi \colon  \tree \setminus T_3 \to \tree \setminus T_4$ be almost automorphisms.
By the previous remark we can choose a finite complete subtree $T \supset T_4 \cup T_1$ of $\tree$ and take representatives for $\psi$ and $\varphi$ with image respectively domain $\tree \setminus T$. These representatives we can compose as maps. The equivalence class of this composition is the product $\varphi \circ \psi$.

\begin{definition}
	With this product the set of almost automorphisms of $\tree$ is a group, called \emph{group of almost automorphisms} of $\tree$ and denoted $\operatorname{AAut}(\tree)$.
\end{definition}

The almost automorphism group was first defined by Neretin \cite{n92} for regular trees, and is then known by the name Neretin's group.

\begin{remark}
	There is an obvious homomorphism $\Aut(\tree) \to \AAut(\tree)$. It is injective if $\tree$ has no leaves. We will therefore consider automorphisms of $\tree$ as almost automorphisms without further notice.
\end{remark}
%
%

\subsection{The topological full group of a groupoid} \label{secttfggroupoid}

The topological full group of a groupoid, generalizing the longer known notion of the topological full group of a group action, was first defined and studied by Matui~\cite{m12}.
We refer to the preliminary section there for a more detailed introduction to it.
See also the survey article \cite{m16}.

\subsubsection{Topological groupoids}
A \index{groupoid} \emph{groupoid} is a small category such that every morphism is an isomorphism. 
Recall that \emph{small} means that the class of objects and the class of morphisms are sets.
A \index{topological groupoid} \emph{topological groupoid} is a groupoid $\mathcal{G}$ such that the set of objects and the set of morphisms are topological spaces and all structure maps (composition, inverse, identity, source and range) are continuous. We will always assume that they are t.d.l.c. and Hausdorff.
We denote by $\mathcal{G}^{(0)}$ the space of objects and by $\mathcal{G}^{(1)}$ the space of morphisms of $\mathcal{G}$.
Denote by $$s,r \colon \mathcal G^{(1)} \to \mathcal{G}^{(0)}$$ the source and range maps.
%

\begin{definition}
	A topological groupoid is called \index{\'etale} \emph{\'etale} if source and range
	are open maps and local homeomorphisms.
\end{definition}

%
%

\begin{definition}
	Let $Y \subset \mathcal{G}^{(0)}$ be a clopen subset. The \index{reduction (groupoid)} \emph{reduction of $\gpoid$ to $Y$} is the full subgroupoid of $\gpoid$ with object set $Y$, denoted $\gpoid|_Y$. Full means that the morphism set is $\{g \in \mathcal{G}^{(1)} \mid s(g) \in Y, r(g) \in Y\}$. Object and morphism sets are endowed with the subspace topology.
\end{definition}
%

\begin{definition}
	A \emph{bisection} of $\mathcal{G}$ is a clopen subset $U \subset \mathcal{G}^{(1)}$
	such that the restricted maps $s|_U \colon U \to \mathcal{G}^{(0)}$ and $r|_U \colon U \to \mathcal{G}^{(0)}$ are homeomorphisms.
\end{definition}

\begin{definition}
	The \index{topological full group (groupoid)} \emph{topological full group} of an \'etale groupoid $\mathcal G$ is 
	$$\mathcal{F}(\mathcal{G}) := \big\{ r \circ (s|_U)^{-1} \in \operatorname{Homeo}(\mathcal{G}^{(0)}) \bigm| U \subset \mathcal G^{(1)} \text{ bisection of } \mathcal{G} \big\}.$$
	We leave it to the reader to check that $\mathcal{F}(\mathcal{G})$ is indeed a subgroup of $\operatorname{Homeo}(\mathcal{G}^{(0)})$.
\end{definition}

\paragraph*{Groupoid Homology.} We will need groupoid homology because it gives isomorphisms between topological full groups, see Theorem \ref{thm:matsthm}. The homology theory for \'etale groupoids was intruduced by Crainic and Moerdijk \cite{cm00}.
For $n \geq 2$ let $$\gpoid^{(n)}:=\{(g_1,\dots,g_n) \in \gpoid^{(1)} \mid r(g_{i+1})=s(g_{i}) \}$$ be the $n$-tuples of composable morphisms.
Define $$d_i \colon \gpoid^{(n)} \to \gpoid^{(n-1)}$$ for $i=0,\dots,n$ by
\[
d_i(g_1,\dots,g_n)=\begin{cases}
(g_2,\dots,g_n) & i=0 \\
(g_1,\dots,g_ig_{i+1},\dots,g_n) & i=1,\dots,n-1 \\
(g_1,\dots,g_{n-1}) & i=n.
\end{cases}
\]
Let $C_c(\gpoid^{(n)},\mathbb{Z})$ be the set of compactly supported continuous functions $\gpoid^{(n)} \to \mathbb{Z}$. It is an abelian group with pointwise addition. For every $f \in C_c(\gpoid^{(n)},\mathbb{Z})$ and every \'etale map $h \colon  X \to Y$ between locally compact Hausdorff spaces we define $h_*(f)(y) := \sum_{h(x)=y} f(x)$.
We now have the chain complex
\[
0 \overset{\delta_0}{\longleftarrow} C_c(\gpoid^{(0)},\mathbb{Z}) \overset{\delta_1}{\longleftarrow} C_c(\gpoid^{(1)},\mathbb{Z})  \overset{\delta_2}{\longleftarrow} C_c(\gpoid^{(3)},\mathbb{Z}) \overset{\delta_3}{\longleftarrow} \dots
\]
with $\delta_1:=s_* - r_*$ and $\delta_n := \sum_{i=0}^{n}(-1)^i {d_i}_*$.
The homology groups are, as usual, defined by $H_n(\gpoid,\mathbb{Z})=\Ker(\delta_n)/\operatorname{Im}(\delta_{n+1})$.

\subsubsection{One-sided irreducible shifts of finite type} \label{chshifts}

We refer to \cite{m15}, Section 6, for a more detailed treatment 
of shifts of finite type and the topological full group associated to them.

\begin{definition}
	Let $\mathfrak{g}=(\mathfrak{v},\mathfrak{e})$ be finite graph.
	The \index{adjacency matrix} \emph{adjacency matrix} of $\mathfrak{g}$ is the matrix $M_\mathfrak{g} \in \mathbb{Z}^{\mathfrak{v} \times \mathfrak{v}}$ such that for all $v,w \in \mathfrak{v}$ the entries of the matrix are
	$$M_\mathfrak{g}(v,w)=\left|\{e \in E' \mid i(e)=v,t(e)=w\}\right|.$$
\end{definition}

\paragraph*{Assumptions on oriented graphs.}
Let $\mathfrak{g}=(\mathfrak{v},\mathfrak{e})$ be a finite graph and $M_\mathfrak{g}$ its adjacency matrix.
We will always impose two conditions on $\frakg$, i.e. on $M_\mathfrak{g}$.
\begin{enumerate}
	\item It must be irreducible, 
	i.e. for all $v,w \in \mathfrak{v}$ there exists an $n>0$ such that $M_\mathfrak{g}^n(v,w)\neq 0$.
	This is equivalent to saying that there exists a path of length $n$ from $v$ to $w$.
	We call such a graph \index{diconnected (oriented graph)} \emph{diconnected}.
	\item It must not be a permutation matrix, 
	which is equivalent to saying that $\mathfrak g$ is not a disjoint union of oriented cycles.
	We call such a graph \index{non-circular} \emph{non-circular}.
\end{enumerate}

\paragraph*{One-sided irreducible shifts of finite type.}
Let $$X_\mathfrak{g} = \big\{(e_k) \in \mathfrak{e}^{\mathbb N} \bigm| \forall k \in \mathbb{N} \colon i(e_{k+1}) = t(e_k) \big\}$$ be the set of all infinite oriented paths in $\mathfrak g$.
Note that $X_\mathfrak{g} \subset \mathfrak{e}^{\mathbb N}$ is closed. 
Moreover, if $\mathfrak{g}$ is diconnected and non-circular then $X_\mathfrak{g}$ is a Cantor space.
Define the map $\sigma \colon X_\mathfrak{g} \to X_\mathfrak{g}$ by $\sigma(e)_k = e_{k+1}$. It is a local homeomorphism. 
The pair $(X_\mathfrak{g},\sigma)$ is called the \index{one-sided shift} \emph{one-sided irreducible shift of finite type associated to $\mathfrak g$}.

\paragraph*{Associated groupoid.}
We associate to $(X_\mathfrak{g},\sigma)$ the following groupoid $\mathcal{G}_\mathfrak{g}$. 
The space of objects and morphisms are
\begin{align*}
\mathcal{G}_\mathfrak{g}^{(0)} &= X_\mathfrak{g} \\
\mathcal{G}_\mathfrak{g}^{(1)} &= \big\{(x,n-m,y) \in X_\mathfrak{g} \times \mathbb Z \times X_\mathfrak{g} \bigm| \sigma^n(x)=\sigma^m(y)\big\} \subset X_\mathfrak{g} \times \mathbb Z \times X_\mathfrak{g}.
\end{align*}
We endow $\mathcal{G}_\mathfrak{g}^{(1)}$ with the topology that is generated by all sets of the form $$\{(x,n-m,y) \mid x \in U,\, y \in V,\, \sigma^n(x)=\sigma^m(y)\}$$ with $U,V \subset X_\mathfrak{g}$ clopen.
The source and range maps, $s$ and $t$, are the projection on the last respectively first factor.
Two elements $(x,n-m,y)$ and $(y',m-l,z)$ are composable if and only if $y=y'$ and the product is $$(x,n-m,y) \cdot (y,m-k,z) = (x,n-k,z).$$
The unit space consists of all elements of the form $(x,0,x)$ and is homeomorphic to $X_\mathfrak{g}$ in an obvious way.
The inverse is given by $(x,n-m,y)^{-1} = (y,m-n,x)$.

\paragraph*{Homology.} The homology can be computed as follows
\[
	H_n(\gpoid_{\frakg},\mathbb{Z}) \cong \begin{cases}
	\Coker(id-M^t) & n=0 \\
	\Ker(id-M^t) & n=1 \\
	0 & n \geq 2,
	\end{cases}
\]
where $M^t$ denotes the transpose of $M$.
The isomorphism between $H_0(\gpoid_{\frakg},\mathbb{Z})$ and $\Coker(id-M^t)$ is induced by the map
$v \mapsto 1_{\{(e_i) \in X_\frakg \mid o(e_0)=v\}}$, see \cite{m15}, beginning of Section 6.2.

\begin{lemma}[\cite{m15} Lemma 5.3] \label{lem:homimage}
	For every $h \in H_0(\gpoid_\frakg,\mathbb{Z})$ there exists a non-empty clopen set $Y \subset X_\frakg$ such that $h=[1_Y]$, where $[1_Y]$ is the homology class of $1_Y$, the indicator function on $Y$.
\end{lemma}

The following theorem of Matui summarizes some properties of these topological full groups.

\begin{theorem}[\cite{m15}, Section 6] \label{shiftthm}
	Let $\mathfrak{g}$ be a finite oriented diconnected and non-circular graph and let
	$M_{\mathfrak{g}}$ be its adjacency matrix. Let $Y \subset X_\frakg$ a clopen subset.
	Then, the topological full group $\mathcal{F}(\mathcal G_\mathfrak{g}|_Y)$ is finitely presented 
	(more precisely, it is of type $F_\infty$).
	Moreover, every non-trivial subgroup of $\mathcal{F}(\mathcal G_\mathfrak{g}|_Y)$ normalized by the commutator subgroup $D(\mathcal{F}(\mathcal G_\mathfrak{g}|_Y))$ contains $D(\mathcal{F}(\mathcal G_\mathfrak{g}|_Y))$.
	In particular $D(\mathcal{F}(\mathcal G_\mathfrak{g}|_Y))$ is simple.
	Its abelianization is independent of $Y$ and isomorphic to
	\[
	\mathcal{F}(\mathcal G_\mathfrak{g}|_Y)/D(\mathcal{F}(\mathcal G_\mathfrak{g}|_Y))  \cong  \left( H_0(\gpoid_\frakg,\mathbb{Z}) \otimes_{\mathbb{Z}} \mathbb{Z}/2\mathbb{Z} \right) \oplus H_1(\gpoid_\frakg,\mathbb{Z}).\]
\end{theorem}
%
%

Our results strongly rely on the following theorem by Matsumoto. The formulation in which we need it is given in \cite{m15}, Theorem 6.2.

\begin{theorem}[\cite{mats10}, Theorem 1.1; \cite{mats13}, Theorem 1.1] \label{thm:matsthm}
  Let $\frakg_1$ and $\frakg_2$ be diconnected and non-circular finite oriented graphs and let $M_1$ respectively $M_2$ be their adjacency matrices. Denote by $\gpoid_{\frakg_1}$ and $\gpoid_{\frakg_2}$ the associated groupoids.
  Suppose that there exists an isomorphism $H_0(\gpoid_{\frakg_1}) \to H_0(\gpoid_{\frakg_2})$ with $[1_{Y_1}] \mapsto [1_{Y_2}]$ and assume moreover that $\det(id-M_1^t)=\det(id-M_2^t)$. Then $\gpoid_{\frakg_1}|_{Y_1}$ is isomorphic as \'etale groupoid to $\gpoid_{\frakg_2}|_{Y_2}$. In particular, their associated topological full groups are isomorphic.
\end{theorem}

As homology in degree $0$ is given by the cokernel, the isomorphism already implies $\det(id-M_1^t)|=|\det(id-M_2^t)|$. It is not known whether the sign of the determinants actually matter.
%
%
%
%

\section{A connection between shifts and almost automorphisms}
\label{ch:shiftalmaut}

In this section we want to show how every topological full group associated to a one-sided shift can be written as a group of tree almost automorphisms. 

\subsection{Label-preserving almost automorphisms}

Now we introduce colourings on trees respecting their symmetries.

\begin{definition} \label{def:patterning}
	Let $\tree=(\vertices,\edges)$ be a rooted tree with root $v_0$.
	For every vetex $v \in \vertices$ denote by $\operatorname{ch}(v) \subset \vertices$ the set of children of $v$.
	A \emph{colouring structure} of $\tree$ is a tuple $\mathscr{C}=(L,\ell,\mathcal{D},\ccol)$ consisting of the following data.
	\begin{itemize}
		\item $L$ is a finite set, called \emph{set of labels}.
		\item $\ell$ is a \emph{labelling} of $\tree$ with label set $L$, that is a map $\ell \colon \vertices \setminus\{v_0\} \to L$ with the following two properties. First, for every two vertices $v,w \in \vertices$ with $\ell(v)=\ell(w)$ there exists a label-preserving bijection $\alpha \colon \operatorname{ch}(v) \to \operatorname{ch}(w)$ such that for every $v' \in \operatorname{ch}(v)$ we have $\ell(\alpha(v'))=\ell(v')$. Second, for every $l \in L$ the set $\ell^{-1}(l)$ is not contained in any proper subtree of $\tree$. 
		\item $\mathcal{D}=\{D_{l,l'} \mid l,l' \in L \}$ is a set of finite disjoint sets, called \emph{set of colours} where for every $l \in L$ the set $D_{l,l'}$ has the cardinality as the set of children with label $l'$ of every vertex with label $l$.
		\item $\ccol$ is a map $\ccol \colon \edges \setminus o^{-1}(v_0) \to \bigsqcup_{l,l' \in L} D_{l,l'}$, called \emph{legal child colouring}, such that for every $v \in \vertices \setminus \{v_0\}$ with $\ell(v)=l$ the restriction $\ccol|_{o^{-1}(v) \cap \ell^{-1}(l')}$ is a bijection $\{e \in \edges \mid o(e)=v, \, \ell(t(v)) =l' \} \to D_{l,l'}$.
	\end{itemize}
    A \emph{patterning} $\mathscr{P}=(L,\ell,\mathcal{D},\ccol,\mathscr{F})$ is a colouring structure $(L,\ell,\mathcal{D},\ccol)$ together with a set of symmetric groups $\mathscr{F}=\{F_l \leq \Sym(\bigsqcup_{l' \in L} D_{l,l'}) \mid l \in L\}$, called the \emph{pattern}, that preserve the partition $D_l := \bigsqcup_{l' \in L} D_{l,l'}$.
    The patterning is called \emph{trivial} if every element of $\mathscr{F}$ is the trivial group.
%
\end{definition}

Note that not every tree admits a colouring structure.
%
%

For trees together with a patterning we now define certain groups of tree almost automorphisms.
The idea for this definition goes back to Burger and Mozes \cite{bm00a} and has since then be used to define many ``locally prescribed" groups.

\begin{definition}\label{defvell}
	Let $\tree$ be a tree with a patterning as in Definition \ref{def:patterning}.
	Let $T_1,T_2$ be complete finite subtrees of $\tree$ and let $\varphi \colon \tree \setminus T_1 \to \tree \setminus T_2$ be an honest almost automorphism of $\tree$.
    The honest almost automorphism $\varphi$ is called \emph{child colour preserving} if and only if for every edge $e \in \edges \setminus o^{-1}(v_0)$ holds $\ccol(e)=\ccol(\varphi(e))$.
	We denote the set of almost automorphisms admitting a child colour preserving representative by $V_{\mathscr{C}}$.
	Let $v$ be a vertex of $\tree \setminus T_1$. The \emph{local permutation} of $\varphi$ at $v$ is the bijection $\prm_{\varphi,v} := \ccol \circ \varphi \circ \ccol^{-1}|_{D_{\ell(v)}} \colon D_{\ell(v)} \to D_{\ell(\varphi(v))}$.
    The honest almost automorphism $\varphi$ is called \emph{$\mathscr{P}$-patterning} if and only if
	\begin{itemize}
		\item for every $v \in \mathcal{L}T_1$ holds $\ell(v)=\ell(\varphi(v))$, and
		\item for every vertex $v$ of $\tree \setminus T_1$ holds $\prm_{\varphi,v} \in F_{\ell(v)}$.
	\end{itemize}
    Note that this implies that for every vertex $v$ of $\tree \setminus T_1$ holds $\ell(v)=\ell(\varphi(v))$.
	We denote the set of all almost automorphisms of $\tree$ admitting a patterning representative by $\AAut_{\mathscr{P}}(\tree)$.
	In addition we write $\Aut_{\mathscr{P}}(\tree):= \AAut_{\mathscr{P}}(\tree) \cap \Aut(\tree)$ and for every finite subtree $T$ of $\tree$ we denote $\Fix_{\mathscr{P}}(T):= \AAut_{\mathscr{P}}(\tree) \cap \Fix(T)$, where $\Fix(T)$ is the group of all tree automorphisms fixing $T$ pointwise.
\end{definition}

Note that $\AAut_{\mathscr{P}}(\tree)$, $\Aut_{\mathscr{P}}(\tree)$ and $V_{\mathscr{C}}$ are, up to conjugation, independent of the choice of $\ccol$.

\begin{remark}
	That $\AAut_{\mathscr{P}}(\tree)$ is a subgroup of $\AAut(\tree)$ is due to the following. For all $\mathscr{P}$-patterning almost automorphisms $\varphi \colon \tree \setminus T_1 \to T_2$ and $\psi \colon \tree \setminus T_3 \to \tree \setminus T_2$ holds $\prm_{\varphi \circ \psi,v} = \prm_{\varphi,\psi(v)} \circ \prm_{\psi,v}$ for all vertices $v$ of $\tree \setminus T_3$.
\end{remark}
Now we can give a definition of the Higman-Thompson groups.

\begin{definition} \label{def:htgroups}
	Let $\tree$ be such that the root has valency $k$ and all other vertices have valency $d+1$. Let $\ell$ be a constant map. Then $V_{\mathscr{C}}$ is called the \emph{Higman-Thompson group}~$V_{d,k}$.
\end{definition}

\begin{example} \label{ex:aautd}
	If $\tree$ is such that the root has valency $k$ and all other vertices have valency $d+1$, and if $\ell$ is a constant map, then $\AAut_{\mathscr{P}}(\tree)$ is a group introduced by Caprace and De Medts in \cite{cm11} and studied by other authors. Le Boudec \cite{lb17} showed this group is compactly presented, a result which was generalized by Sauer and Thuman \cite{st17}.
\end{example}

\begin{remark} \label{rem:constrpattel}
	An element $\varphi \colon \tree \setminus T_1 \to \tree \setminus T_2$ in $\AAut(\tree)$ it is uniquely determined by the following data:
	\begin{itemize}
		\item the bijection $\varphi|_{\mathcal{L}T_1} \colon \mathcal{L}T_1 \to \mathcal{L}T_2$
		\item the element $\prm_{\varphi,v}$ for every vertex $v$ of $\tree \setminus T_1$.
	\end{itemize}
    Moreover $\varphi$ is $\mathscr{P}$-patterning if and only if for all but finitely many vertices $v$ holds $\prm_{\varphi,v} \in F_{\ell(v)}$.
\end{remark}

The following lemma is known for the groups in Example \ref{ex:aautd}, see \cite{lb17} Section 4. Since its proof is standard, we will provide only an outline.

\begin{lemma} \label{lem:basicpropaautp}
	The set $\AAut_{\mathscr{P}}(\tree)$ is a subgroup of $\AAut(\tree)$. The set $\Aut_{\mathscr{P}}(\tree)$ is a closed subgroup of $\Aut(\tree)$, which is discrete if any only if the pattern is trivial. There exists a unique group topology on $\AAut_{\mathscr{P}}(\tree)$ such that $\Aut_{\mathscr{P}}(\tree)$ is an open subgroup. With this topology $V_{\mathscr{C}}$ is dense in $\AAut_{\mathscr{P}}(\tree)$.				
\end{lemma}

\begin{proof}
	It is easy to see that $\AAut_{\mathscr{P}}(\tree)$ is a subgroup of $\AAut(\tree)$ and that $\Aut_{\mathscr{P}}(\tree)$ is a closed subgroup of $\Aut(\tree)$.
	Let $T$ be a simple expansion of the subtree of $\tree$ consisting only of the root.
	If the pattern is trivial, then an easy induction argument shows that the intersection $\Fix_{\mathscr{P}}(\tree) \cap \Aut(\tree)$ is trivial, hence $\Aut_{\mathscr{P}}(\tree)$ is discrete.
	Let on the other hand $\mathscr{P}$ be non-trivial, and let $l \in L$ be such that $F_l$ contains a non-trivial element $\tau$. Let $T$ be a finite complete subtree of $\tree$. By making $T$ bigger if necessary we can assume that $T$ has a vertex $v$ with $\ell(v)=l$.
	Then by Remark \ref{rem:constrpattel} we can define an element $g$ of $\Fix_{\mathscr{P}}(T)$ by
	\[
	\prm_{g,w}=\begin{cases}
	\tau & w=v \\
	id & \text{ else.}
	\end{cases}
	\]
	Clearly, $g$ is non-trivial, hence $\Aut_{\mathscr{P}}(\tree)$ is non-discrete.
	
	For the topology, let again $T$ be a finite complete subtree. Let $\varphi \in \AAut_{\mathscr{P}}(\tree)$.
	By Bourbaki \cite{b98}, Chapter III, Prop. 1, it is enough to show that $\varphi \Fix_{\mathscr{P}}(T) \varphi^{-1}$ contains $\Fix_{\mathscr{P}}(T')$ for some finite subtree $T'$ of $\tree$. Let $T_1,T_2$ be finite complete subtrees of $\tree$ such that $\varphi$ has a representative $\varphi \colon \tree \setminus T_1 \to \tree \setminus T_2$, and assume that $T \subset T_1$. Observe that $\varphi \Fix_{\mathscr{P}}(T_1) \varphi^{-1}=\Fix_{\mathscr{P}}(T_2)$, so $\Fix_{\mathscr{P}}(T_2) \leq \varphi \Fix_{\mathscr{P}}(T) \varphi^{-1}$, and we are done.
	
	For the last statement, let again $T$ be a finite subtree of $\tree$. Let $\varphi \in \AAut_{\mathscr{P}}(\tree)$ and let $T_1,T_2$ be complete finite subtrees of $\tree$ such that $\varphi$ has a representative $\varphi \colon \tree \setminus T_1 \to \tree \setminus T_2$. Assume without loss of generality that $T=T_1$. Let $\psi \in V_{\mathscr{C}}$ be such that it has a representative $\psi \colon \tree \setminus T_2 \to \tree \setminus T_1$ with $\psi|_{\mathcal{L}T_2}=(\varphi|_{\mathcal{L}T_1})^{-1}$ and for every vertex $v$ of $\tree \setminus T_2$ holds $\prm_{\psi,v}=id$.
	Then clearly $\psi \circ \varphi \in \Fix_{\mathscr{P}}(T)$, hence $V_{\mathscr{C}}$ is dense.
\end{proof}

\subsection{From shifts to almost automorphisms}

In this section $\frakg=(\mathfrak{v},\mathfrak{e})$ denotes a diconnected and non-circular finite oriented graph.
As in Section \ref{secttfggroupoid} we denote the groupoid associated to it by $\gpoid_\frakg$ and its topological full group by $\tfg(\gpoid_\frakg)$.
The goal of this section is to express $\tfg(\gpoid_\frakg)$ as a group of tree almost automorphisms.

The intuition behind is the following. The group $\tfg(\gpoid_\frakg)$ is a group of homeomorphisms of $X_\frakg$, which is a set of infinite paths. Also, a group of almost automorphisms of a tree $\tree$ can be viewed as a group of homeomorphisms of a set of infinite paths, namely of the boundary $\partial \tree$. It turns out that if we define an appropriate tree together with a colouring structure with the vertex set being the set of finite paths of $\frakg$, then $X_\frakg$ can be identified with the boundary of this tree and $\tfg(\gpoid_\frakg)$ corresponds precisely to the set of label-preserving almost automorphisms.

\begin{theorem}\label{thmshifttoalmaut}
	Let $\frakg=(\mathfrak{v},\mathfrak{e})$ be a diconnected and non-circular finite oriented graph. Let $Y \subset X_\frakg$ be a clopen subset.
	There exists a rooted tree $\tree=(\vertices,\edges)$ with root $v_0$ and a colouring structure $\mathscr{C}$ on $\tree$ such that the group of colour-preserving almost automorphisms $V_{\mathscr{C}}$ satisfies $$\tfg(\gpoid_\frakg|_Y) \cong V_{\mathscr{C}},$$
	i.e. it is isomorphic to the topological full group of the reduction to $Y$ of the groupoid associated to the one-sided shift  of~$\frakg$.
\end{theorem}

\begin{remark} \label{remtreeconstr}
	We now construct a tree and a colouring structure and will then prove that it satisfies the conclusions of Theorem \ref{thmshifttoalmaut}. We call $\tree$ together with the colouring $\mathscr{C}=(\mathfrak{v},\ell,\mathcal{D},\ccol)$ the \emph{unfolding tree} of the pair $(\frakg,Y)$
	
	Assume first that $Y=X_\frakg$.
	Denote the set of all finite oriented paths in $\frakg$ by
	$\mathcal{P}:=\{(e_0,\dots,e_n) \mid n \in \mathbb{N}, \, e_i \in \mathfrak{e}, \, t(e_i)=o(e_{i+1}) \} \cup \mathfrak{v}$.
	Let $\vertices := \mathcal{P} \cup \{\emptyset\}$.
	We consider $\emptyset$ the root of $\tree$. The children of $\emptyset$ are the paths of length $0$ in $\frakg$, i.e. the elements of $\mathfrak{v}$. Now let $\gamma \in \vertices$ be a finite path in $\mathfrak{g}$. Then the set of children of $\gamma$ is the set of all paths that continue $\gamma$ one step further. Formally, denote by $t(\gamma) \in \mathfrak{v}$ the endpoint of $\gamma$.
	If $\gamma=(e_0,\dots,e_n) \in \mathfrak{e}^{n+1}$ then for every $e$ with $o(e)=t(e_n)$ we simply write $(\gamma,e):=(e_0,\dots,e_n,e)$.
	Then, the set of children of $\gamma$ is $\operatorname{ch}(\gamma) := \{(\gamma,e) \mid e \in \mathfrak{e}, \, o(e)=t(\gamma)\}$. Denote the edge connecting $\gamma$ with $(\gamma,e)$ by $\gamma * e$.
	This completes the definition of the tree $\tree$.
	The set of labels is $\mathfrak{v}$ and the labelling $\ell$ is now simply defined by $\ell(\gamma)=t(\gamma)$ for a finite path $\gamma \in \vertices$. For the child colouring we set $D_{l,l'}=\{e \in \mathfrak{e} \mid o(e)=l, \, t(e)=l' \}$ and $\ccol(\gamma * e)=e$.
	%
	
	Note that $\partial \tree$ can be identified with $X_\frakg$ in an obvious way.
	Now if $Y$ is a proper clopen subset of $X_\frakg$, we can under this identification say that $Y \subset \partial \tree$ and replace $\tree$ with its minimal subtree with boundary $Y$.
\end{remark}

\begin{proof}[Proof of Theorem \ref{thmshifttoalmaut}]
	Let $\tree = (\vertices,\edges)$ be the unfolding tree of $(\frakg,Y)$ as in Remark~\ref{remtreeconstr}.
	Since $\frakg$ is diconnected, each vertex of $\tree$ has for every $v \in \mathfrak{v}$ a descendant with label $v$, so it is clear that $\ell^{-1}(v)$ is not contained in a proper subtree of $\tree$.
	
	The set of infinite paths $Y$ can be identified with the set of infinite paths $\partial  \tree$ in an obvious way. We will first show that under this identification every element of $\tfg(\gpoid_\frakg|_Y)$ is indeed in $V_{\mathscr{C}}$, and then we will prove the converse.
	
	Let $U \subset \mathcal{G}_{\mathfrak{g}}|_Y$ be a bisection.
	Then there exist clopen partitions $\{U_1,\dots,U_n\}$ and $\{U'_1,\dots,U'_n\}$ of $Y$ and 
	positive integers $n_1,\dots,n_n,m_1,\dots,m_n$ satisfying the following.
	The bisection $U$ can be written as $$U = \bigsqcup_{k=1}^n U'_{k} \times \{n_k - m_k \} \times U_k$$
	and $r \circ (s|_{U})^{-1}$ restricts to a homeomorphism $U_k \to U'_k$ for every $1 \leq k\leq n$.
	By making the $U_i$ and $U'_i$ smaller if necessary, we can assume that for every $k$ there exist finite paths $\gamma_k=(e_{k0},\dots,e_{kn_k})$ and
	$\gamma_k'=(e'_{k0},\dots,e'_{km_k})$
	such that
	\begin{align*}
		U_k &= \{(e_i)_{i \in \mathbb{N}} \in Y \mid \forall 0 \leq i \leq n_k \colon e_i = e_{ki} \} \\
		U'_k &= \{(e'_i)_{i \in \mathbb{N}} \in Y \mid \forall 0 \leq i \leq m_k \colon e'_i = e'_{ki} \}
	\end{align*}
	and
	$$r \circ (s|_{U})^{-1}(e_{k0},\dots,e_{kn_k},e_{n_k+1},e_{n_k+2},\dots)
	= (e'_{k0},\dots,e'_{km_k},e_{n_k+1},\dots).$$
	Informally, $U_k$ (resp. $U_k'$) consists of all paths with initial segment $\gamma_k$ (resp. $\gamma_k'$).
	%
	Observe that for $k=0,\dots,n$ holds
	\begin{align*}
		U_k = \partial \tree_{\gamma_k} \quad \text{ and } \quad 
		U'_k = \partial \tree_{\gamma_k'}.
	\end{align*}
	Since $\{U_1,\dots,U_n\}$ is a clopen partition of $Y$, there exists a unique finite complete subtree $T \subset \tree$ with 
	$$\mathcal{L}T = \{\gamma_k \mid k = 1,\dots,n\}.$$
	In the same way there exists a unique finite complete subtree $T' \subset \tree$ with 
	$$\mathcal{L}T' = \{\gamma_k' \mid k = 1,\dots,n\}.$$
	The element $r \circ (s|_{U})^{-1}$ is then the almost automorphism $\tree \setminus T \to \tree \setminus T'$ such that for every vertex 
	$\gamma = (e_{k0},\dots,e_{kn_k},e_{n_k+1},e_{n_k+2},\dots,e_{i'})$ 
	%
	of $\tree \setminus T$ holds 
	$$r \circ (s|_{U})^{-1}(\gamma)=(e'_{k0},\dots,e'_{km_k},e_{n_k+1},e_{n_k+2},\dots,e_{i'}).$$
	This shows $\tfg(\gpoid_\frakg|_Y) \subseteq V_{\mathscr{C}}$.
	
	The reverse inclusion works in the same way, just backwards.
	Let $T,T' \subset \tree$ be finite complete subtrees and let
	$\varphi \colon \tree \setminus T \to \tree \setminus T'$ be a child colour preserving honest almost automorphism. 
	Let $\gamma_k=(e_{k0},\dots,e_{kn_k})$ and
	$\gamma_k'=(e'_{k0},\dots,e'_{km_k})$ be be the leaves of $T$ resp. $T'$, i.e.
	$\mathcal{L}T=\{\gamma_1,\dots,\gamma_k\}$ and $\mathcal{L}T'=\{\gamma_1',\dots,\gamma_k'\}$, and assume that for $i=1,\dots,k$ holds $\varphi(\gamma_i)=\gamma_i'$. Define
	\begin{align*}
		U_k &= \partial \tree_\gamma = \{(e_i)_{i \in \mathbb{N}} \in Y \mid \forall 0 \leq i \leq n_k \colon e_i = e_{ki} \} \\
		U'_k &= \partial \tree_{\gamma'} = \{(e'_i)_{i \in \mathbb{N}} \in Y \mid \forall 0 \leq i \leq m_k \colon e'_i = e'_{ki} \}.
	\end{align*}
	Then
	$$U = \bigsqcup_{k=1}^n U'_{k} \times \{n_k - m_k \} \times U_k$$
	is a bisection such that $r \circ (s|_U)^{-1}=\varphi$.
\end{proof}

\begin{example}
	This generalizes the observation of Matui \cite{m15}, Remark 6.3, based on a result by Nekrashevych \cite{n04}, Proposition 9.6, that the Higman--Thompson group $V_{d,k}$ is isomorphic to $\tfg(\gpoid_\frakg)$ for the graph $\frakg$ as in Figure \ref{fig:ht}.	
	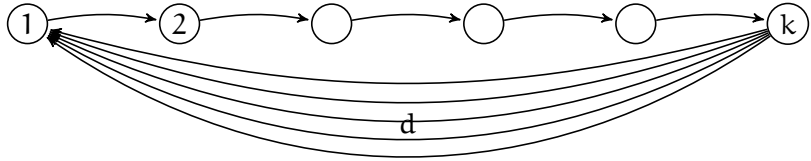
\begin{figure}[H]
		\centering
		\begin{tikzpicture}[->,>=stealth',shorten >=1pt,auto,node distance=2cm,
			semithick]
			
			\node[state,inner sep=2pt,minimum size=0pt]        (A)              {$1$};
			\node[state,inner sep=2pt,minimum size=0pt]        (B) [right of=A] {$2$};
			\node[state,inner sep=2pt,minimum size=0pt]        (C) [right of=B] {$\phantom{1}$};
			\node[state,inner sep=2pt,minimum size=0pt]        (D) [right of=C] {$\phantom{1}$};
			\node[state,inner sep=2pt,minimum size=0pt]        (E) [right of=D] {$\phantom{1}$};
			\node[state,inner sep=2pt,minimum size=0pt]        (F) [right of=E] {$k$};
			
			\path (A) edge [bend left=10]    (B)
			(B) edge [bend left=10]    (C)
			(C) edge [bend left=10]    (D)
			(D) edge [bend left=10]    (E)
			(E) edge [bend left=10]    (F)
			(F) edge [bend left=15]    (A)
			(F) edge [bend left=20] node {$d$}    (A)
			(F) edge [bend left=25]    (A)
			(F) edge [bend left=30]    (A)
			(F) edge [bend left=35]    (A);
		\end{tikzpicture}
		\caption{Matui's graph $\mathfrak{g}$ such that $V_{d,k}=\mathcal{F}(\mathcal{G}_{\mathfrak{g}})$}
		 \label{fig:ht}
	\end{figure}
\end{example}

\subsection{Commensurated subgroups and the local prime content}

\begin{definition}
	Let $G$ be a group.
	Two subgroups $U$ and $V$ of $G$ are called \emph{commensurate} if $U\cap V$ has finite index in both $U$ and $V$.
	Being commensurate is an equivalence relation.
	The set $[U] := \{V \leq G \mid U, V \text{ are commensurate}\}$ is called the \emph{commensurability class} of $U$.
	An element $g \in G$ is said to \emph{commensurate} the subgroup $U$ if $U$ and $gUg^{-1}$ are commensurate.
	A subgroup $U$ is called \emph{commensurated} if every $g \in G$ commensurates $U$.
\end{definition}

Basic examples of commensurated subgroups are finite subgroups, finite index subgroups and normal subgroups.
Other important examples are compact open subgroups of a t.d.l.c. group. Note, however, that a profinite group can have non-closed finite index subgroups, so it is not true in general that the set of compact open subgroups of a t.d.l.c. group is the whole commensurability class of one compact open subgroup.

\begin{lemma} \label{lem:cosrep}
	Let $G$ and $H$ be topological groups and let $\varphi \colon H \to G$ be a continuous homomorphism with dense image. Let $O \leq G$ be an open subgroup.
	Let $\{h_i \mid i \in I\}$ be a set of coset representatives of $\varphi^{-1}(O) \leq H$.
	Then, $\{\varphi(h_i) \mid i \in I\}$ is also a set of coset representatives of $O \leq G$.
	In particular $[G:O]=[H:\varphi^{-1}(O)]$.
\end{lemma}

\begin{proof}
	Note that for $i \neq j$ holds $h_i h_j^{-1} \notin \varphi^{-1}(O)$, so the $\varphi(h_i)$ represent different cosets of $O$ in $G$. Also note that at most one of the $h_i$ can be in the kernel of $\varphi$ as the kernel is contained in $\varphi^{-1}(O)$.
	
	Let $g \in G$. Since $\varphi(H)$ is dense and $O$ is open, there exists an $h \in H$ such that $g \in \varphi(h)O$.
	Let $i \in I$ be such that $h_i \varphi^{-1}(O) = h \varphi^{-1}(O)$. This implies $\varphi(h_i) O = \varphi(h) O$. In particular, $g \in \varphi(h_i) O$ and therefore $G = \bigsqcup_{i \in I} \varphi(h_i) O$.
\end{proof}

The following definition is due to Gl\"ockner \cite{g06}.

\begin{definition}
	Let $G$ be a totally disconnected, locally compact group. The \emph{local prime content} of $G$, denoted $\lpc(G)$, is the set of all primes $p$ such that for every compact open $U \leq G$ there exists an open subgroup $V \leq U$ with $p \mid [U:V]$. 
\end{definition}

In the introduction we defined the local prime content differently, but the versions are equivalent, see \cite{g06}, p. 449.
We chose to give this formulation now so that we can extend the definition to commensurability classes..

\begin{definition}
	Let $H$ be a topological group and $K$ a commensurated subgroup of $H$.
	The \emph{local prime content} of the commensurability class $[K]$, denoted $\lpc([K])$, is the set of all primes $p$ such for every $L \in [K]$ there exists an $L' \in [K]$ with $L' \leq L$ and $p \mid [L:L']$.
\end{definition}

\begin{proposition} \label{prop:lpccontained}
	Let $H$ be a topological group, $G$ a t.d.l.c. group and let $\varphi \colon H \to G$ be a continuous homomorphism with dense image. Let $U \leq G$ be compact and open. Then $\lpc(G) \subseteq \lpc([\varphi^{-1}(U)])$.
\end{proposition}

\begin{proof}
	Let $K=\varphi^{-1}(U)$.
	Let $p \in \lpc(G)$. Let $L \in [K]$. By replacing $L$ with $K \cap L$ we can assume that $L \leq K$. Since $\varphi(L)$ has finite index in $\varphi(K)$, the closure $\overline{\varphi(L)}$ has finite index in $U=\overline{\varphi(K)}$, hence $\overline{\varphi(L)}$ is compact and open. There exists a compact and open $V \leq \overline{\varphi(L)}$ with $p \mid [\overline{\varphi(L)}:V]$. By Lemma~\ref{lem:cosrep}, applied to $\varphi|_L \colon L \to \overline{\varphi(L)}$, holds $[\overline{\varphi(L)}:V] = [L : \varphi^{-1}(V) \cap L]$. In particular $\varphi^{-1}(V) \cap L$ is commensurate to $L$ and the index is divisible by $p$. Hence $p$ is in the local prime content of $[\varphi^{-1}(U)]$.
\end{proof}

\section{Dense embeddings of topological full groups}

In this section we use the re-interpretation of topological full groups coming from one-sided shifts as tree almost automorphism groups to embed these topological full groups densely into t.d.l.c. groups. We use the local prime content to distinguish these t.d.l.c. groups from one another.

For the following lemma, recall from Definition \ref{def:patterning} that a pattern $\mathscr{F}$ is a set of permutation groups indexed by labels.

\begin{lemma} \label{lem:locpermquot}
	Let $\tree=(\vertices,\edges)$ be a tree, let $\mathscr{C}=(L,\ell,\mathcal{D},\ccol)$ be a colouring structure on $\tree$ and let $\mathscr{P}=(L,\ell,\mathcal{D},\ccol,\mathscr{F})$ be a patterning. Let $T$ be a finite complete subtree of $\tree$ and let $T'$ be a simple expansion of $T$ at a vertex $v \in \mathcal{L}T$. Then $$\Fix(T)/\Fix(T') \cong \Fix(T) \cap V_{\mathscr{C}} /\Fix(T') \cap V_{\mathscr{C}} \cong F_{\ell(v)}.$$
\end{lemma}

\begin{proof}
	The map $\Fix(T) \cap V_{\mathscr{C}} \to \Fix(T)/\Fix(T')$ is a surjective homomorphism because $\Fix(T')$ is open, and clearly its kernel is $\Fix(T') \cap V_{\mathscr{C}}$. This shows the first isomorphism.
	
	The map $\Fix(T) \to F_{\ell(v)}$ defined by $g \mapsto \prm_{g,v}$ clearly has kernel $\Fix(T')$. It is a surjective homomorphism, since by Remark \ref{rem:constrpattel} every $\tau \in F_{\ell(v)}$ has a preimage defined by
	\[
	\prm_{g,w}=\begin{cases}
	\tau & w=v \\
	id & \text{ else.}
	\end{cases}
	\]
\end{proof}

Now we determine the local prime content of groups of patterning tree almost automorphisms.

\begin{proposition} \label{prop:lpcisproduct}
	Let $\tree=(\vertices,\edges)$ be a tree, let $\mathscr{C}=(L,\ell,\mathcal{D},\ccol)$ be a colouring structure on $\tree$ and let $\mathscr{P}=(L,\ell,\mathcal{D},\ccol,\mathscr{F})$ be a patterning. Let $U \leq \AAut_{\mathscr{P}}(\tree)$ be compact and open.
	Then $\lpc(\AAut_{\mathscr{P}}(\tree)) = \lpc([U \cap V_{\mathscr{C}}])$, and this local prime content equals the set of prime factors of the integer $\prod_{F \in \mathscr{F}}|F|$.
\end{proposition}

\begin{proof}
	By Proposition \ref{prop:lpccontained} we have $\lpc(\AAut_{\mathscr{P}}(\tree)) \subseteq \lpc([U \cap V_{\mathscr{C}}])$.
	
	Let $p$ be a prime factor of $\prod_{F \in \mathscr{F}}|F|$. Let $l$ be a label such that $p\mid |F_l|$. By making $U$ smaller if necessary we can assume that $U=\Fix_{\mathscr{P}}(T)$ for a finite complete subtree $T$ of $\tree$. We can furthermore assume that one of the leaves of $T$ has label $l$ since $\ell^{-1}(l)$ is not contained in a proper subtree of $\tree$. Now Lemma \ref{lem:locpermquot} shows that $p \in \lpc(\AAut_{\mathscr{P}}(\tree))$.
	
	Let on the other hand $p$ be a prime number which does not divide $\prod_{F \in \mathscr{F}}|F|$.
	Let $T$ be the simple expansion of the finite subtree of $\tree$ consisting only of the root.
	Assume by contradiction that there exists a group $O \leq \Fix_{\mathscr{P}}(T) \cap V_{\mathscr{C}}$ such that $[\Fix_{\mathscr{P}}(T) \cap V_{\mathscr{C}}:O]$ is finite and divisible by $p$. By replacing $O$ by its finitely many $(\Fix_{\mathscr{P}}(T) \cap V_{\mathscr{C}})$-conjugates we can assume it is normal in $\Fix_{\mathscr{P}}(T) \cap V_{\mathscr{C}}$. Take an element $y \in \Fix_{\mathscr{P}}(T) \cap V_{\mathscr{C}}$ such that its coset $yO$ in $(\Fix_{\mathscr{P}}(T) \cap V_{\mathscr{C}})/O$ has order $p$.
	Then, because every element of $\Fix_{\mathscr{P}}(T) \cap V_{\mathscr{C}}$ has finite order, a power $x$ of $y$ has order $p$.
	Let $v$ be a leaf of $T$. Then $x$ acts on the children of $v$. But since it must act like an element of $F_{\ell(v)}$ and its order does not divide $|F_{\ell(v)}|$, this implies that $x$ acts trivially on the children of $v$. Inductively, we get that $x$ is the trivial element, which is a contradiction.
%
\end{proof}

This proposition allows to distinguish some groups of patterning tree almost automorphisms from others. Still, it would be nice to know more precise conditions under which they are, or are not, isomorphic.

\begin{theorem} \label{thm:mainthm}
	Let $\frakg$ be a diconnected, non-circular oriented finite graph. Let $Y \subset X_\frakg$ be a clopen subset.
	Let $P$ be a finite set of primes.
	
	Then, there exists a t.d.l.c. completion $\varphi \colon \tfg(\gpoid_\frakg|_Y) \to G$ of $\tfg(\gpoid_\frakg|_Y)$ such that for every compact open subgroup $U \leq G$ holds $\lpc(G)=\lpc([\varphi^{-1}(U)])=P$.
	
	More precisely, there exists a tree $\tree$ with a colouring structure $\mathscr{C}=(L,\ell,\mathcal{D},\ccol)$ 
	with $\tfg(\gpoid_\frakg|_Y)\cong V_{\mathscr{C}}$
	and a patterning $\mathscr{P}=(L,\ell,\mathcal{D},\ccol,\mathscr{F})$ such that $P$ is the set of prime divisors of $\prod_{F \in \mathscr{F}} |F|$.
\end{theorem}

\begin{proof}
	Let $M$ be the adjacency matrix of $\frakg$. The following claim is a generalization of the discussion after Theorem 6.12 in \cite{m15}.
	
	\emph{Claim:} There exists a matrix $A$ over $\mathbb{Z}$ with the following properties:
	\begin{itemize}
		\item $\det(A)= \det(id-M^t)$
		\item $\Coker(A) \cong \Coker(id-M^t)$
		\item $id-A^t$ is the adjacency matrix of a diconnected non-circular finite oriented graph with the property that one of its entries is $N := \prod_{p \in P}p$.
	\end{itemize}
	
	We construct $A$ as follows. Let $d_1 \geq \dots \geq d_n \geq 0$ be integers such that $\Coker(id-M^t) \cong \bigoplus_{i=1}^n \mathbb{Z}/d_i \mathbb{Z}$.
	Let 
	Let $A_0$ be one of the two diagonal matrix $\operatorname{diag}(-1,-d_1,\dots,-d_n)$ or $\operatorname{diag}(-1,-1,-d_1,\dots,-d_n)$, namely the one which has the same determinant as $id-M^t$.
	Add the first row of $A_0$ to all other rows to get a new matrix $A_1$.
	To obtain $A$, add $N$ times the first column of $A_1$ to all its other columns.
	Thus,
	\[
	A = \left( \begin{array}{ccccc}
	           -1 & -N     & -N  & \dots &  -N \\
	           -1 & -d_1-N & -N  & \dots &  -N \\
	           -1 & -N & -d_2-N  & \dots &  -N \\
	           \vdots & \vdots  & \vdots & \ddots  & \vdots \\
	           -1 & -N & -N  &  \dots  & -d_n-N
	           \end{array}  \right) 
	\]
	or
	\[
	A = \left( \begin{array}{ccccc}
	-1 & -N     & -N  & \dots &  -N \\
	-1 & -N-1     & -N  & \dots &  -N \\
	-1 & -N & -d_1-N  & \dots &  -N \\
	\vdots & \vdots  & \vdots & \ddots  & \vdots \\
	-1 & -N & -N  &  \dots  & -d_n-N
	\end{array}  \right) 
	\]
	satisfies the claim.
	
	Let $\widetilde{M}=id-A^t$. Let $\widetilde{\frakg}$ be a finite oriented graph with adjacency matrix $\widetilde{M}$. Recall from the discussion above Lemma \ref{lem:homimage} that $\Coker(id-M^t)\cong H_0(\gpoid_\frakg)$, so there exists an isomorphism $\varphi \colon H_0(\gpoid_{\frakg}) \to H_0(\gpoid_{\widetilde{\gpoid_\frakg}})$.
	By Lemma \ref{lem:homimage} there exists a clopen subset $\widetilde{Y}$ of $X_{\widetilde{\frakg}}$ such that $\varphi([1_Y])=[1_{\widetilde{Y}}]$. By Theorem \ref{thm:matsthm} the topological full groups of $\gpoid_{\frakg}|_Y$ and $\gpoid_{\widetilde{\frakg}}|_{\widetilde{Y}}$ are isomorphic.
	Let $\tree$ be the unfolding tree of $\widetilde{\frakg}$ and $\widetilde{Y}$ together with the colouring structure $\mathscr{C}=(L,\ell,\mathcal{D},\ccol)$ as in Remark \ref{remtreeconstr}. Let $\mathscr{P}=(L,\ell,\mathcal{D},\ccol,\mathscr{F})$ be a patterning such that $P$ is the set of prime divisors of $\prod_{F \in \mathscr{F}} |F|$. For example, take $v, w$ vertices of $\widetilde{\frakg}$ such that there are $N$ many edges from $v$ to $w$, let $F_{v}$ be generated by a product of disjoint $p$-cycles for every $p \in P$, and let $F_{w}$ be the trivial group for every label $w \neq v$. Theorem \ref{thmshifttoalmaut} implies that $\tfg(\gpoid_{\widetilde{\gpoid_\frakg}}|_{\widetilde{Y}}) \cong V_{\mathscr{C}}$.
	 By Proposition \ref{prop:lpcisproduct} we know that $\lpc(\AAut_{\mathscr{P}}(\tree))=P$, and by Lemma \ref{lem:basicpropaautp} the group $V_{\mathscr{C}}$ is a dense subgroup of $\AAut_{\mathscr{P}}(\tree)$.
\end{proof}

\begin{remark}
	Let $\frakg$ be a diconnected, non-circular oriented graph.
	Note that every subgroup $G \leq \Aut(\frakg)$ fixing all vertices induces a patterning $\mathscr{P}$ on the unfolding tree $\tree$. The local prime content of $\AAut_{\mathscr{P}}(\tree)$ is then the set of prime factors of $|G|$.
\end{remark}

For Thompson's group $V$ we can give much nicer graphs as constructed in the proof.

\begin{example}\label{ex:v}
	For Thompson's group $V$ and prime number $p$ one can take a graph $\frakg$ with the following adjacency matrix
	\[
	M_p = \left( \begin{array}{cc}
	1 & 1 \\
	1 & p
	\end{array} \right).
	\]
	It is easily verified that $\det(id-M_p^t)=-1$. Therefore, by Theorem \ref{thm:matsthm} the topological full group $\tfg(\gpoid_\frakg)$ is isomorphic to $V$.
	Taking the group generated by graph automorphisms that fix all vertices, but act like a $p$-cycle on the outgoing leaves of the vertex that has $p$ loops, one obtains a completion of $V$ with local prime content $\{p\}$.
	\begin{figure}[H]
		\centering
		\begin{tikzpicture}[->,>=stealth',shorten >=1pt,node distance=10cm,
		semithick]
		
		\def \r {2cm}
		
		\node[state,inner sep=2pt,minimum size=0pt]        (1)    at (0,0)     {};
		\node[state,inner sep=2pt,minimum size=0pt]        (2)    at (-2,0)    {};
		
		\path 
		
		(1) edge [bend right=20]    (2)
		(2) edge [bend right=20]    (1)
		
		(2) edge [loop left, looseness=50] (2)
		
		(1) edge [loop right, looseness=30] (1)
		(1) edge [loop right, looseness=40] (1)
		(1) edge [loop right, looseness=50] node {$\cdots$} (1)
		(1) edge [loop right, looseness=90] (1)
		(1) edge [loop right, looseness=100] (1);
		\node[] at (1,-0.5) {$p$ loops};
		\end{tikzpicture}
		\caption{A graph $\mathfrak{g}$ for prime content $\{p\}$}
	\end{figure}
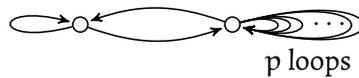
	
	For not necessarily different prime numbers $p_1,\dots,p_n$ with $n \geq 2$ one can take a graph $\frakg$ with the following adjacency matrix
	\[
	M(i,j)= \begin{cases}
	p_i & j = i+1\mod n \\
	\prod_{k=1}^n p_k - 2 & i=j=1 \\
	0 & \text{else.}
	\end{cases}
	\]
	Again it is easily verified that $\det(id-M_p^t)=-1$ and by Theorem \ref{thm:matsthm} the topological full group $\tfg(\gpoid_\frakg)$ is isomorphic to $V$.
	Taking the group generated by graph automorphisms that fix all vertices and loops, but for every $i$ act like a $p_i$-cycle on the outgoing leaves of the vertex $i$, one obtains a completion of $V$ with local prime content $\{p_1,\dots,p_n\}$.
	
	\begin{figure}[H]
		\centering
		\begin{tikzpicture}[->,>=stealth',shorten >=1pt,node distance=10cm,
		semithick]
		
		\def \r {2cm}
		
		\node[state,inner sep=2pt,minimum size=0pt]        (1)         at (0:\r)           {$1$};
		\node[state,inner sep=2pt,minimum size=0pt]        (2)  at (72:\r)   {$2$};
		\node[state,inner sep=2pt,minimum size=0pt]        (3)  at (144:\r) {$3$};
		\node[state,inner sep=2pt,minimum size=0pt]        (4)  at (216:\r) {$4$};
		\node[state,inner sep=2pt,minimum size=0pt]        (5)  at (288:\r) {$5$};
		
		\path 
		
		(1) edge [bend right=10]    (2)
		(1) edge [bend right=20]    (2)
		
		(2) edge [bend right=10]    (3)
		(2) edge [bend right=20]    (3)
		(2) edge [bend right=30]    (3)
		
		(3) edge [bend right=10]    (4)
		(3) edge [bend right=20]   (4)
		(3) edge [bend right=30]    (4)
		(3) edge [bend right=40]   (4)
		(3) edge [bend right=50]    (4)
		
		(4) edge [bend right=10] (5)
		(4) edge [bend right=20] (5)
		(4) edge [bend right=30] (5)
		(4) edge [bend right=40] (5)
		(4) edge [bend right=50] (5)
		(4) edge [bend right=60] (5)
		(4) edge [bend right=70] (5)
		
		(5) edge [bend right=10] (1)
		(5) edge [bend right=20] (1)
		(5) edge [bend right=30] (1)
		(5) edge [bend right=40] (1)
		(5) edge [bend right=50] (1)
		(5) edge [bend right=60] (1)
		(5) edge [bend right=70] (1)
		(5) edge [bend right=80] (1)
		(5) edge [bend right=90] (1)
		(5) edge [bend right=100] (1)
		(5) edge [bend right=110] (1)
		
		(1) edge [loop right, looseness=10] (1)
		(1) edge [loop right, looseness=15] (1)
		(1) edge [loop right, looseness=20] (1)
		(1) edge [loop right, looseness=25] (1)
		(1) edge [loop right, looseness=30] (1)
		(1) edge [loop right, looseness=35] node {$\cdots$} (1)
		(1) edge [loop right, looseness=55] (1)
		(1) edge [loop right, looseness=60] (1)
		(1) edge [loop right, looseness=65] (1)
		(1) edge [loop right, looseness=70] (1)
		(1) edge [loop right, looseness=75] (1);
		\node[] at (4.5,-1) {$2308$ loops};
		\end{tikzpicture}
		\caption{A graph $\mathfrak{g}$ for prime content $\{2,3,5,7,11\}$}
	\end{figure}
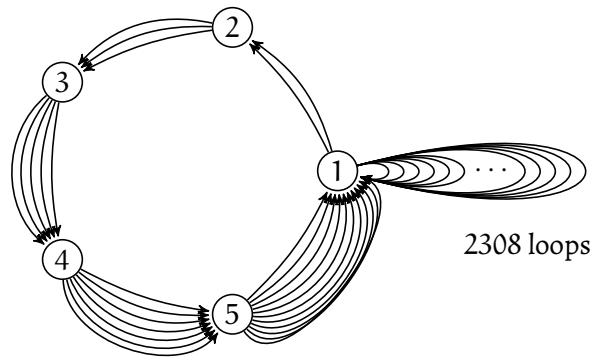
\end{example}

The next proposition shows in particular that the completions we construct in Theorem \ref{thm:mainthm} are topologically simple if $V_{\mathscr{C}}$ is simple. For a group $G$ we denote its commutator subgroup by $D(G)$.

\begin{proposition}\label{prop:simple}
	Let $\tree$ be a tree with colouring structure $\mathscr{C}=(L,\ell,\mathcal{D},\ccol)$ 
	and a patterning $\mathscr{P}=(L,\ell,\mathcal{D},\ccol,\mathscr{F})$.
	Then $\overline{D(\AAut_{\mathscr{P}}(\tree))} = \overline{D(V_\mathscr{C})}$, and this group is topologically simple.
\end{proposition}

\begin{proof}
	It is clear that $\overline{D(\AAut_{\mathscr{P}}(\tree))} \supseteq \overline{D(V_\mathscr{C})}$.
	For the reverse inclusion, observe that by density of $V_\mathscr{C}$, see Lemma \ref{lem:basicpropaautp}, for all $g,h \in \AAut_{\mathscr{P}}(\tree)$ there exist sequences $(g_n)$ and $(h_n)$ in $V_\mathscr{C}$ converging to $g$ and $h$ respectively. Then $[g_n,h_n]$ converges to $[g,h]$, hence $D(\AAut_{\mathscr{P}}(\tree)) \subseteq \overline{D(V_\mathscr{C})}$ and the equality follows.
	
	The group $\AAut_{\mathscr{P}}(\tree)$, considered as discrete group, is its own topological full group for the \'etale action groupoid of itself acting on $\partial \tree$. Hence by \cite{m15}, Theorem 4.16, the group $D(\AAut_{\mathscr{P}}(\tree))$ is simple (note that the assumptions countability and essential freeness from the quoted theorem by Matui are not satisfied in our case, but these assumptions are not used in its proof).
	
	Let now $N \unlhd \overline{D(\AAut_{\mathscr{P}}(\tree))}$ be a proper closed normal subgroup. By simplicity $N$ must intersect $D(\AAut_{\mathscr{P}}(\tree))$ trivially. Therefore, $N$ has to be abelian. In particular, every nontrivial element $n \in N$ has to commute with all its $D(\AAut_{\mathscr{P}}(\tree))$-conjugates. We will show that this is not possible.
	
	\emph{Case 1:} The element $n$ has order at least $3$.
	
	Let $x \in \partial \tree$ be such that $n^{-1}(x), x$ and $n(x)$ are different. Let $U \subset \partial \tree$ be a clopen neighbourhood of $x$ and $g,h \in \AAut_{\mathscr{P}}(\tree)$ satisfying the following
	\begin{itemize}
		\item $U, n(U), n^{-1}(U), g(U), g^{-1}(n(U)), h(n(U))$ are pairwise disjoint
		\item $g|_U = h|_U$
		\item $g|_{n^{-1}(U)} = h|_{n^{-1}(U)}$
		\item $h|_{g^{-1}(n(U))} = id$.
	\end{itemize}
    Then $\varphi := [g,h]$ satisfies
    \begin{itemize}
    	\item $\varphi|_U = id$
    	\item $\varphi|_{n^{-1}(U)} = id$
    	\item $\varphi|_{n(U)}=h|_{n(U)}$.
    \end{itemize}
    Now observe that $n \cdot \varphi n \varphi^{-1} \neq  \varphi n \varphi^{-1} \cdot n$, because $n[\varphi,n](x) \in n(U)$ and $\varphi n \varphi^{-1}(x) \in h(U)$.
    
    \emph{Case 2:} The element $n$ has order $2$.
    
    Let $x \in \partial \tree$ be such that $x$ and $n(x)$ are different. Let $U$ be a clopen neighbourhood of $x$ and $g,h \in \AAut_{\mathscr{P}}(\tree)$ satisfying the following
	\begin{itemize}
	\item $U, g(U), h^{-1}(U), h^{-1}(g(U)), n(U), g(n(U)), n(g(U)), g(n(g)), h^{-1}(n(g(U)))$ are pairwise disjoint
	\item $g(g(U))=U$
	\item $g|_{h^{-1}(U)}=id$
	\item $g|_{h^{-1}(g(U))}=id$
	\item $g|_{n(U)}=h|_{n(U)}$
	\item $g|_{h^{-1}(n(g(U)))}=id$.
\end{itemize}
Then $\varphi := [g,h]$ satisfies
\begin{itemize}
	\item $\varphi(U)=g(U)$
	\item $\varphi(\varphi(U))=U$
	\item $\varphi|_{n(U)} = id$
	\item $\varphi(n(\varphi(U)))=g(U)$.
\end{itemize}
Now observe that $n \cdot \varphi n \varphi^{-1} \neq  \varphi n \varphi^{-1} \cdot n$, because $n[\varphi,n](x) \in g(n(g(U)))$ and $\varphi n \varphi^{-1}(x) \in n(g(U))$.
\end{proof}

 \bibliographystyle{alpha}
 \bibliography{references}
\end{document}